\documentclass[a4paper,10pt]{amsart}
\usepackage[utf8]{inputenc}
\usepackage{amsmath,amsfonts,amssymb,mathrsfs,latexsym}
\usepackage[all,cmtip]{xy}
\usepackage[top=1.5cm, bottom=2cm, left=2.5cm, right=2.5cm]{geometry}
\newtheorem{definition}{Definition}[section]
\newtheorem{theorem}[definition]{Theorem}
\newtheorem{lemma}[definition]{Lemma}
\newtheorem{coro}[definition]{Corollary}
\newtheorem{proposition}[definition]{Proposition}
\newtheorem{remark}[definition]{Remark}

\begin{document}
\title{Quaternionic Projective Bundle Theorem and Gysin Triangle in MW-Motivic Cohomology}
\author{Nanjun Yang}
\thanks{This work has been partially supported by ERC ALKAGE.}
\address{Institut Fourier-UMR 5582 \\ Universit\'e Grenoble-Alpes \\ CS 40700 \\ 38058 Grenoble Cedex 9 \\ France}
\email{ynj.t.g@126.com}


\begin{abstract}
In this paper, we show that the motive of the quaternionic Grassmannian \(HP^n\) (as defined by I. Panin and C. Walter) splits in the category of effective MW-motives (as defined by B. Calm\`es, F. D\'eglise and J. Fasel). Moreover, we extend this result to an arbitrary symplectic bundle,  obtaining the so-called quaternionic projective bundle theorem. Finally, we give the Gysin triangle in MW-motivic cohomology.
\end{abstract}

\maketitle

\tableofcontents

\section{Introduction}
Thoughout, we denote by \(k\) the base field and \(pt:=\textrm{Spec }k\).

The aim of this paper is to investigate the fundamental properties of MW-motivic cohomology, as defined by B. Calm\`es, F. D\'eglise and J. Fasel. This cohomology theory is a generalization of ordinary motivic cohomology as developed by V. Voevodsky. One of the basic properties of the latter is the projective bundle theorem (see \cite[2.10]{D}, \cite[Theorem 15.12]{MVW}, \cite[Theorem 4.5]{SV}:
\begin{proposition}\label{pbt}
Let \(X\) be a smooth scheme over a perfect field \(k\) and \(\mathscr{E}\) be a vector bundle of rank \(n\) over \(X\), then the map
\[\xymatrixcolsep{5pc}\xymatrix{\mathbb{Z}_{pt}(\mathbb{P}(\mathscr{E}))\ar[r]^-{p\boxtimes c_1(O_E(1))^i}&\oplus_{i=0}^{n-1}\mathbb{Z}_{pt}(X)(i)[2i]}\]
is an isomorphism in \(DM^{eff,-}(pt)\), where \(p:\mathbb{P}(\mathscr{E})\longrightarrow X\) is the structure map of the projective bundle of \(E\) and \(c_1\) is the first Chern class map.
\end{proposition}
It is the motivic version of the same theorem of Chow rings (see \cite[A.2]{H}) and implies the existence of Chern classes of vector bundles in Chow rings. Here \(DM^{eff,-}(pt)\) (see \cite[Definition 14.1]{MVW}) is the Voevodsky's category of effective motives, obtained by \(\mathbb{A}^1\)-localization of the (bounded above) derived category of Nisnevich sheaves with tranfers and \(\mathbb{Z}_{pt}(X)\) is the motive of \(X\) in \(DM^{eff,-}(pt)\).

Proposition \ref{pbt} is a general consequence of the fact that motivic cohomology is a \(GL\)-oriented cohomology theory (see \cite{P1} and \cite{PS}) in the sense of \cite[Definition 3.3]{A2}. In constrast, MW-motivic cohomology fails to be \(GL\)-oriented, so one cannot expect a projective bundle formula as usual (see Remark \ref{pw}). Thus it's natual to ask whether there are analogues for \(Sp\)-oriented or \(SL\)-oriented cohomology theories. In cohomological terms, there are some results in literature. For example, \cite{PW} (resp. \cite{A}) has shown the quaternionic projective bundle theorem (resp. special linear projective bundle theorem) for \(Sp\)-oriented cohomology theories (resp. \(SL\)-oriented cohomology theories with invertible hopf map). The Chow-Witt groups of \(BSp_{2n}\) and \(BSL_n\) are also calculated in \cite{HW}. Our work is to give the motivic analogue of the quaternionic projective bundle theorem in \cite{PW} (see Theorem \ref{quaternionic projective bundle theorem}) over any smooth base in the case of MW-motivic cohmology:
\begin{theorem}\label{quaternionic projective bundle theorem1}
Let \(S\) be a smooth scheme over an infinite perfect field of characteristic different from \(2\), \(X\) being smooth over \(S\) and \((\mathscr{E},m)\) be a symplectic vector bundle of rank \(2n+2\) on \(X\). Then, the map
\[\xymatrixcolsep{6pc}\xymatrix{\widetilde{\mathbb{Z}}_S(HGr_X(\mathscr{E}))\ar[r]^-{\pi\boxtimes p_1(\mathscr{U}^{\vee})^i}&\oplus_{i=0}^{n}\widetilde{\mathbb{Z}}_S(X)(2i)[4i]}\]
is an isomorphism in \(\widetilde{DM}^{eff,-}(S)\), where \(HGr_X(\mathscr{E})\) is the quaternionic projective bundle of \(\mathscr{E}\) (see Definition \ref{quaternionic projective bundle}), parametrizing rank two symplectic subbundles of \(\mathscr{E}\), \(\pi:HGr_X(\mathscr{E})\longrightarrow X\) is its structure map, \(\mathscr{U}^{\vee}\) is the dual tautological bundle and \(p_1\) is the first Pontryagin class map (see Definition \ref{first Pontryagin class}).
\end{theorem}
In the statement of the theorem, \(\widetilde{DM}^{eff,-}(S)\) is the category of effective MW-motives over \(S\) as defined in \cite[\S 3.2]{DF} and \cite[Definition 6.3]{Y}, obtained by \(\mathbb{A}^1\)-localization of the (bounded above) derived category of Nisnevich sheaves with MW-tranfers over \(S\) and \(\widetilde{\mathbb{Z}}_S(X)\) is the motive of \(X\) in \(\widetilde{DM}^{eff,-}(S)\).

Moreover, we provide a Gysin triangle (Theorem \ref{gysin triangle}) for MW-motivic cohomology following the method of  \cite{A2}. Recall that the Gysin triangle in Voevodsky's category of effective motives is of the following form (\cite[Theorem 15.15]{MVW} and \cite[Theorem 4.10]{SV}).

\begin{proposition}\label{g}
Let \(X\) be a smooth scheme over a perfect field and \(Y\subseteq X\) be a smooth closed subscheme with \(\textrm{codim}(Y)=n\). Then we have a distinguished triangle in \(DM^{eff,-}(pt)\):
\[\mathbb{Z}_{pt}(X\setminus Y)\longrightarrow \mathbb{Z}_{pt}(X)\longrightarrow \mathbb{Z}_{pt}(Y)(n)[2n]\longrightarrow \mathbb{Z}_{pt}(X\setminus Y)[1].\]
\end{proposition}

In the case of MW-motivic cohomology, the Gysin triangle is a little bit more complicated:

\begin{theorem}
Let \(S\) be a smooth scheme over an infinite perfect field of characteristic different from \(2\), \(X, Y\) being smooth over \(S\) and \(Y\) be a closed subscheme of \(X\) with \(\textrm{codim}(Y)=n\).
\begin{enumerate}
\item
{If \(detN_{Y/X}\cong O_Y\), we have a distinguished triangle
\[\widetilde{\mathbb{Z}}_S(X\setminus Y)\longrightarrow\widetilde{\mathbb{Z}}_S(X)\longrightarrow \widetilde{\mathbb{Z}}_S(Y)(n)[2n]\longrightarrow\widetilde{\mathbb{Z}}_S(X\setminus Y)[1]\]
in \(\widetilde{DM}^{eff,-}(S)\).}
\item
If \(S=pt\) and {\(n>0\)}, we have a distinguished triangle
\[\widetilde{\mathbb{Z}}_S(X\setminus Y)\longrightarrow\widetilde{\mathbb{Z}}_S(X)\longrightarrow Th_S(det(N_{Y/X}))(n-1)[2n-2]\longrightarrow\widetilde{\mathbb{Z}}_S(X\setminus Y)[1]\]
in \(\widetilde{DM}^{eff,-}(S)\), where the third term only depends (up to an isomorphism) {on \(n\)} and the class of \(det(N_{Y/X})\) in \(Pic(Y)/2Pic(Y)\).
\end{enumerate} 
\end{theorem}

{Here for every vector bundle \(E\) over \(X\), \(Th_S(E)\) is the Thom space (see Definition \ref{thomdef}) of \(E\).}

The organization of the paper is as follows. In Section \ref{MWcomplexes}, we briefly survey the properties of the category \(\widetilde{DM}^{eff,-}(S)\) for the convenience of the reader, which all have analogues in \(DM^{eff,-}(S)\). Here, our exposition slightly differs from the one in \cite{DF}, avoiding altogether the notion of model category. We also recall the definition and basic properties of MW-motivic cohomology. In Section \ref{Grass}, we recall the definition of the quaternionic Grassmannian and basic constructions of the symplectic Thom structure, while Section \ref{sl} is devoted to explaining the \(SL\)-orientation of Chow-Witt rings. The proof of the quaternionic projective bundle theorem is in Section \ref{Projbundle}. The proof of the Gysin triangle in the last section concludes this paper.

\textbf{Conventions. }
\begin{enumerate}
\item All the schemes are over an infinite perfect field \(k\) of characteristic not \(2\) unless specified. The field $k$ is called the base field. We denote the category of smooth and separated schemes by \(Sm/k\). For any \(S\in Sm/k\), we denote the category of smooth and {separated} schemes over \(S\) by \(Sm/S\).
\item For any category \(\mathscr{C}\) and \(X,Y\in\mathscr{C}\), we set \(\mathscr{C}(X,Y)=Hom_{\mathscr{C}}(X,Y)\).
\item We always use the notation \(S^c\) to denote the complement of the subset \(S\) in some set.
\end{enumerate}

\section{MW-motivic complexes over smooth bases}\label{MWcomplexes}

In this section, we recall the basic definitions and facts about the category of MW-motives over smooth bases following \cite{CD}, \cite{CF}, \cite{DF} and \cite{Y} (and sometimes \cite{MVW} and \cite{SV} when we appeal to properties of the category of ordinary motives).

\subsection{Sheaves with MW-transfers}
Let {\(m\in\mathbb{Z}\)}, \(F/k\) be a finitely generated field extension of the base field \(k\) and \(L\) be a one-dimensional $F$-vector space. One can define \(K^{MW}_{ m}(F,L)\) as in \cite[Remark 2.21]{Mo}. If \(X\) is a smooth scheme, \(\mathscr{L}\) is a line bundle over \(X\) and \(y\in X\), we set 
\[
\widetilde{K}^{MW}_{ m}(k(y),\mathscr{L}):=K^{MW}_{ m}(k(y),\Lambda_y^*\otimes_{k(y)}\mathscr{L}_y),
\]
where \(k(y)\) is the residue field of \(y\) and \(\Lambda_y^*\) is the exterior product of the tangent space of \(y\). If \(T\subset X\) is a closed set, \(n\in\mathbb{N}\), define
\[C_{RS,T}^n(X;\underline{K}_m^{MW};\mathscr{L})=\bigoplus_{y\in X^{(n)}\cap T}\widetilde{K}^{MW}_{m-n}(k(y),\mathscr{L}),\]
where \(X^{(n)}\) means the points of codimension \(n\) in \(X\). Then \(C_{RS,T}^*(X;\underline{K}_m^{MW};\mathscr{L})\) form a complex (see \cite[Definition 4.11]{Mo}, \cite[Remark 4.13]{Mo}, \cite[Theorem 4.31]{Mo} and \cite[D\'efinition 10.2.11]{F1}), which is called the Rost-Schmid complex with support on \(T\). The readers may refer to \cite[p. 18-19]{Y} for a detailed treatment of differential maps of that complex. Define (see \cite[Definition 3.1]{CF})
\[\widetilde{CH}^n_T(X,\mathscr{L})=H^n(C^*_{RS,T}(X;\underline{K}_n^{MW};\mathscr{L})).\]

Suppose \(S\in Sm/k\). For any \(X, Y\in Sm/S\) (recall our conventions on smooth schemes), define \(\mathscr{A}_S(X,Y)\) to be the poset of closed subset in \(X\times_SY\) such that each of its component is finite over a connected component of \(X\) and of dimension \(dimX\). Let
\[\widetilde{Cor}_S(X,Y):=\varinjlim_T\widetilde{CH}_T^{dimY-dimS}(X\times_SY,\omega_{X\times_SY/X})\]
be the finite Chow-Witt correspondences between \(X\) and \(Y\) over \(S\), where \(T\in\mathscr{A}_S(X,Y)\). For any \(f\in\widetilde{Cor}_S(X,Y)\) and \(g\in\widetilde{Cor}_S(Y,Z)\), we can define \(g\circ f\in\widetilde{Cor}_S(X,Z)\) as in \cite[4.2]{CF}. This produces an additive category (see \cite[Proposition 5.1]{Y} and \cite[Proposition 5.5]{Y} for complete details) \(\widetilde{Cor}_S\) whose objects are the same as in \(Sm/S\) and whose morphisms are defined above. There is a functor \(\widetilde{\gamma}:Sm/S\longrightarrow\widetilde{Cor}_S\) sending a morphism to its graph, more precisely, to the push-forward of the identity along the graph morphism (see \cite[4.3]{CF}, \cite[Definition 5.3]{Y}).

We define a presheaf with MW-transfers to be a contravariant additive functor from \(\widetilde{Cor}_S\) to \(Ab\). It's a sheaf with MW-transfers if it's a Nisnevich sheaf after restricting to \(Sm/S\) via \(\widetilde{\gamma}\). For any smooth scheme $X$, let $\widetilde{c}_S(X)$ be the representable presheaf with MW-transfers of \(X\).

Let  \(\widetilde{PSh}(S)\) be the category of presheaves with MW-transfers over \(S\) and let \(\widetilde{Sh}(S)\) be the full subcategory of sheaves with MW-transfers (see \cite[Definition 1.2.1 and Definition 1.2.4]{DF}). Both categories are abelian and have enough injectives (\cite[\S 1.1, Proposition 1.2.11]{DF}). There is an adjunction (see \cite[Proposition 1.2.11]{DF})
\[\widetilde{a}:\widetilde{PSh}(S)\rightleftharpoons\widetilde{Sh}(S):\widetilde{O},\]
where \(\widetilde{a}\) is the sheafication functor and \(\widetilde{O}\) is the forgetful functor. We set 
\[\widetilde{\mathbb{Z}}_S(X)=\widetilde{a}(\widetilde{c}_S(X)).\]

For any \(F\in\widetilde{PSh}(S)\) and \(T\in Sm/S\), we define a presheaf with MW-transfers  \(F^T\)  by \(F^T(X)=F(X\times_S T)\) following \cite[Exercise 2.9]{MVW}. For any \(f:T_1\longrightarrow T_2\) in \(\widetilde{Cor}_S\), we have a morphism \(F^f:F^{T_2}\longrightarrow F^{T_1}\) induced by the tensor product of correspondences (see \cite[4.4]{CF} and \cite[Definition 5.7]{Y}). It's clear that if \(F\) is a sheaf with MW-transfers, \(F^T\) is a sheaf with MW-transfers as well. For any presheaf with MW-transfers $F$, we define a complex \(C_*F\) with \((C_*F)_n=F^{\triangle^n}, n\geq 0\) with usual boundary maps for (co-)simplicial complexes, where \(\triangle^n\) is the algebraic \(n\)-simplex (see \cite[Definition 2.14]{MVW} for details).

The following definitions comes from \cite[Lemma 2.1]{SV}.
\begin{definition}
Let $n\geq 2$ and let  \(F_i, G\in\widetilde{PSh}(S)\) for \(i=1,\cdots,n\). A multilinear function \(\varphi:F_1\times\cdots\times F_n\longrightarrow G\) is a collection of multilinear maps of abelian groups
\[\varphi_{(X_1,\cdots,X_n)}:F_1(X_1){\times\cdots\times} F_n(X_n)\longrightarrow G(X_1\times_S\cdots\times_S X_n)\]
for every \(X_i\in Sm/S\), such that for every \(f\in\widetilde{Cor}_S(X_i,X_i')\), we have a commutative diagram
\[
	\xymatrix
	{
		\cdots\times F_i(X_i')\times\cdots\ar[r]^-{\varphi(\cdots,X_i',\cdots)}\ar[d]_{\cdots\times F(f)\times\cdots}	&G(\cdots\times_S X_i'\times_S\cdots)\ar[d]_{G(\cdots\times f\times\cdots)}\\
		\cdots\times F_i(X_i)\times\cdots\ar[r]^-{\varphi(\cdots,X_i,\cdots).}								&G(\cdots\times_S X_i\times_S\cdots)
	}
\]
\end{definition}

\begin{definition}
Let $n\geq 2$ be an integer and let \(F_i, G\in\widetilde{PSh}(S)\) (resp. \(\widetilde{Sh}(S)\)) for \(i=1,\cdots,n\). The tensor product \(F_1\otimes_S^{pr}\cdots\otimes_S^{pr}F_n\) (resp. \(F_1\otimes_S\cdots\otimes_SF_n\)) is the presheaf (resp. sheaf) with MW-transfers \(G\) such that for any \(H\in\widetilde{PSh}(S)\) (resp. \(\widetilde{Sh}(S)\)), we have
\[Hom_S(G,H)\cong\{\textrm{Multilinear functions }F_1\times\cdots\times F_n\longrightarrow H\}\]
naturally.
\end{definition}
For any \(F, G\in\widetilde{PSh}(S)\), we can define \(F\otimes^{pr}_SG\in\widetilde{PSh}(S)\) as in the discussion before \cite[Lemma 2.1]{SV}, which has the universal property above. Moreover, we define \(\underline{Hom}_S(F,G)\) to be the presheaf with MW-transfers which sends \(X\in Sm/S\) to \(Hom_S(F,G^X)\). And if they are sheaves with MW-transfers, we define \(F\otimes_SG=\widetilde{a}(F\otimes^{pr}_SG)\). If \(G\) is a sheaf with MW-transfers, it's clear that \(\underline{Hom}_S(F,G)\) is also a sheaf with MW-transfers. Finally, it's clear that \(\otimes^{pr}_S\) (resp. \(\otimes_S\)) gives \(\widetilde{PSh}(S)\) (resp. \(\widetilde{Sh}(S)\)) a symmetric monoidal structure.
\begin{proposition}\label{hom-tensor}
For any \(F,G,H\in\widetilde{PSh}(S)\), we have isomorphisms
\[Hom_S(F\otimes^{pr}_SG,H)\cong Hom_S(F,\underline{Hom}_S(G,H)),\]
\[Hom_S(F\otimes^{pr}_SG,H)\cong Hom_S(G,\underline{Hom}_S(F,H))\]
being functorial in three variables. Similarly, for any \(F,G,H\in\widetilde{Sh}(S)\), we have isomorphisms
\[Hom_S(F\otimes_SG,H)\cong Hom_S(F,\underline{Hom}_S(G,H)),\]
\[Hom_S(F\otimes_SG,H)\cong Hom_S(G,\underline{Hom}_S(F,H))\]
being functorial in three variables.
\end{proposition}
\begin{proof}
This is clear from the definition of the bilinear map.
\end{proof}
\begin{proposition}\label{sheafication}
If a morphism \(f:F_1\longrightarrow F_2\) of presheaves with MW-transfers becomes an isomorphism after sheafifying, then so does the morphism \(f\otimes^{pr}_SG\) for any presheaf with MW-transfers \(G\).
\end{proposition}
\begin{proof}
The condition is equivalent to the map \(Hom_S(f,H)\) is an isomorphism between abelian groups for any sheaf with MW-transfers \(H\). And
\[Hom_S(f\otimes^{pr}_SG,H)\cong Hom_S(f,\underline{Hom}_S(G,H))\]
by the proposition above.
\end{proof}

For every \(f\in\widetilde{Cor}_S(X,Y)\), there is a natural map \(\widetilde{\mathbb{Z}}_S(f):\widetilde{\mathbb{Z}}_S(X)\longrightarrow\widetilde{\mathbb{Z}}_S(Y)\) induced by \(f\). For every \(X\in Sm/S\) and \(x:S\longrightarrow X\), we say that the pair \((X,x)\) is a pointed scheme. We define \(\widetilde{\mathbb{Z}}_S((X_1,x_1)\wedge\ldots\wedge(X_n,x_n))\) for pointed schemes \((X_i,x_i)\) as the cokernel of the map
\[\xymatrixcolsep{8pc}\xymatrix{\theta_n:\oplus_i\widetilde{\mathbb{Z}}_S(X_1\times_S\ldots\times_S\widehat{X_i}\times_S\ldots\times_S X_n)\ar[r]^-{\sum(-1)^{i-1}id\times\ldots\times x_i\times\ldots\times id}&\widetilde{\mathbb{Z}}_S(X_1\times_S\ldots\times_S X_n)}.\]
We denote \(\widetilde{\mathbb{Z}}_S((X,x)\wedge\ldots\wedge(X,x))\) by \(\widetilde{\mathbb{Z}}_S((X,x)^{\wedge n})\), \(\widetilde{\mathbb{Z}}_S((X,x))\) by \(\widetilde{\mathbb{Z}}_S((X,x)^{\wedge1})\) and \(\widetilde{\mathbb{Z}}_S(\textrm{Spec }k)\) by \(\widetilde{\mathbb{Z}}_S((X,x)^{\wedge0})\).

As usual, we define \(\widetilde{\mathbb{Z}}_S(q)=C_*\widetilde{\mathbb{Z}}_S(\mathbb{G}_m^{\wedge q})[-q]\) (\(\mathbb{G}_m:=(\mathbb{A}^1\setminus 0,1)\)) for \(q\geq 0\) and further we set \(\widetilde{\mathbb{Z}}_S=\widetilde{\mathbb{Z}}_S(0)\).

Now suppose \(f:S\longrightarrow T\) is a morphism in \(Sm/k\). For every \(F\in\widetilde{Sh}(S)\), we define \(f_*F\) by \((f_*F)(Y)=F(Y\times_TS)\) for any \(Y\in Sm/T\).
\begin{proposition}
\begin{enumerate}
\item For any \(X, Y\in Sm/S\), we have
\[\widetilde{\mathbb{Z}}_S(X)\otimes_S\widetilde{\mathbb{Z}}_S(Y)\cong\widetilde{\mathbb{Z}}_S(X\times_S Y).\]
\item There is an adjoint pair
\[f^*:\widetilde{Sh}(T)\rightleftharpoons\widetilde{Sh}(S):f_*,\]
where \(f^*\) is monoidal and for any \(Y\in Sm/T\), \(f^*\widetilde{\mathbb{Z}}_T(Y)=\widetilde{\mathbb{Z}}_S(Y\times_ST)\).
\item If \(f\) is smooth, there is an adjoint pair
\[f_{\#}:\widetilde{Sh}(T)\rightleftharpoons\widetilde{Sh}(S):f^*,\]
where for any \(X\in Sm/S\), \(f_{\#}\widetilde{\mathbb{Z}}_S(X)=\widetilde{\mathbb{Z}}_T(X)\) and for any \(F\in\widetilde{Sh}(S)\) and \(G\in\widetilde{Sh}(T)\), we have
\[f_{\#}(F\otimes_Sf^*G)\cong(f_{\#}F)\otimes_TG.\]
\end{enumerate}
\end{proposition}
\begin{proof}
\begin{enumerate}
\item By using Proposition \ref{sheafication} since we have \(\widetilde{c}_S(X)\otimes_S\widetilde{c}_S(Y)\cong\widetilde{c}_S(X\times_S Y)\).
\item See \cite[2.5.2]{D1}.
\item See \cite[2.5.3 and Lemma 2.17]{D1} or \cite[Proposition 5.25]{Y}.
\end{enumerate}
\end{proof}
\begin{proposition}\label{MV}
Let \(X\in Sm/S\) and \(U_1\cup U_2=X\) be a Zariski covering. Then, we have an exact sequence of sheaves with MW-transfers:
\[0\longrightarrow\widetilde{\mathbb{Z}}_S(U_1\cap U_2)\longrightarrow\widetilde{\mathbb{Z}}_S(U_1)\oplus\widetilde{\mathbb{Z}}_S(U_2)\longrightarrow\widetilde{\mathbb{Z}}_S(X)\longrightarrow 0.\]
\end{proposition}

\begin{proof}
The proof of \cite[Proposition 6.14]{MVW} applies, replacing \cite[Proposition 6.12]{MVW} by \cite[Lemma 1.2.6]{DF} or \cite[Proposition 5.10]{Y}.
\end{proof}

\begin{definition}\label{thomdef}
Let \(X\in Sm/S\) and \(Y\subseteq X\) be a closed subset. Consider the quotient sheaf with MW-transfers
\[\widetilde{M}_Y(X):=\widetilde{\mathbb{Z}}_S(X)/\widetilde{\mathbb{Z}}_S(X\setminus Y).\]
It's called the relative motive of \(X\) with support in \(Y\) (see \cite[Definition 2.2]{D} and the remark before \cite[Corollary 5.3]{SV}). If \(E\) is a vector bundle over a \(X\), we define the Thom space \(Th_S(E)\) of \(E\) as \(\widetilde{M}_X(E)\) via the zero section.
\end{definition}
\begin{proposition}\label{excision}
(\'Etale excision, see \cite[Lemma 4.11]{SV}) Let \(f:X\longrightarrow Y\) be an \'etale morphism in \(Sm/S\), \(Z\subseteq Y\) be a closed subset of \(Y\) such that the map \(f:f^{-1}(Z)\longrightarrow Z\) is an isomorphism (the schemes are endowed with their reduced structure), then the map \(\widetilde{M}_{f^{-1}(Z)}(X)\longrightarrow\widetilde{M}_Z(Y)\) is an isomorphism of sheaves with MW-transfers.
\end{proposition}
\begin{proof}
By the condition given, we get a Nisnevich covering \(f\amalg id:X\amalg(Y\setminus Z)\longrightarrow Y\) of \(Y\). So we have a commutative diagram with exact (after sheafications) rows and columns by \cite[Lemma 1.2.6]{DF}:
\[
	\xymatrix
	{
																		&0\ar[d]										&0\ar[d]										&\\
																		&\widetilde{c}_S(Y\setminus Z)\ar@{=}[r]\ar[d]			&\widetilde{c}_S(Y\setminus Z)\ar[d]					&\\
		\widetilde{c}_S((X\amalg(Y\setminus Z))\times_Y(X\amalg(Y\setminus Z)))\ar[r]\ar[rd]^-r	&\widetilde{c}_S(X\amalg(Y\setminus Z))\ar[r]^-{p}\ar[d]	&\widetilde{c}_S(Y)\ar[r]\ar[d]						&0\\
																		&\widetilde{c}_S(X)\ar[r]^-q\ar[d]					&\widetilde{c}_S(Y)/\widetilde{c}_S(Y\setminus Z)\ar[d]\ar[r]	&0\\
																		&0											&0											&
	}.
\]

We want to show that \(ker(q)=\widetilde{c}_S(X\setminus f^{-1}(Z))\) after sheafication yielding the statement.

We clearly have \(\widetilde{c}_S(X\setminus f^{-1}(Z))\subseteq ker(q)\) and \(r\) maps onto \(ker(q)\) after sheafication. So it suffices to show that \(Im(r)\subseteq\widetilde{c}_S(X\setminus f^{-1}(Z))\). The sheaf \(\widetilde{c}_S((X\amalg(Y\setminus Z))\times_Y(X\amalg(Y\setminus Z)))\) is decomposed into four direct components
\[\widetilde{c}_S(X\times_YX),\widetilde{c}_S(X\times_Y(Y\setminus Z)),\widetilde{c}_S((Y\setminus Z)\times_YX),\widetilde{c}_S((Y\setminus Z)\times_Y(Y\setminus Z))\]
via disjoint unions so we just have to calculate their images under \(r\) respectively. The calculations for last three components are easy and we only explain the computation of the first one.

We have a Cartesian square
\[
	\xymatrix
	{
		X\times_YX\ar[r]^-{p_1}\ar[d]_{p_2}\ar[rd]^{\pi}	&X\ar[d]_f\\
		X\ar[r]^{f}								&Y.
	}
\]
Then for any \(x\in\pi^{-1}(Z)\), \(p_1(x)=p_2(x)\) and the morphisms \(k(p_1(x))\longrightarrow k(x)\) induced by \(p_1\) and \(p_2\) are equal since \(f^{-1}(Z)\cong Z\). So by \cite[Corollary 3.13]{M}, \(p_1=p_2\) on the connected component containing \(x\). Hence \(p_1=p_2\) on a closed and open set \(U\) containing \(\pi^{-1}(Z)\). Again, \(\widetilde{c}_S(X\times_YX)=\widetilde{c}_S(U)\oplus\widetilde{c}_S(U^c)\). So we have \(r|_{\widetilde{c}_S(U)}=0\) and \(Im(r|_{\widetilde{c}_S(U^c)})\subseteq\widetilde{c}_S(X\setminus f^{-1}(Z))\). So we have proved that \(Im(r)\subseteq\widetilde{c}_S(X\setminus f^{-1}(Z))\).
\end{proof}

\subsection{Motivic complexes}

Suppose \(S\in Sm/k\). Let \(D^-(S)\) (resp. \(D(S)\)) be the derived category of the category $C^-(\widetilde{Sh}(S))$ of bounded above (resp. unbounded) complexes of sheaves with MW-transfers (\cite[\S 10.4]{W}) with the classical trianglated structure. The category \(D(S)\) is also the homotopy category of the model structure in \cite[Theorem 1.7]{CD} with the triangulated structure induced by the model structure.

{The} following proposition summarizes what we need about (Nisnevich) hypercohomology. 

\begin{proposition}\label{prop:relative}
For any $C\in D^-(S)$ and any $i\in \mathbb{N}$, we have an isomorphism of functors {$Sm/S\to \mathcal{A}b$}
\[Hom_{D^-(S)}(\widetilde{\mathbb{Z}}_S(-),C[i])\cong\mathbb{H}^i(-,C).\]
Further, let $X\in Sm/S$, $Z\subset X$ be a closed subset and $U=X\setminus Z$. Then, we have an isomorphism of functors  $D^{-}(S)\to  \mathcal{A}b$
\[Hom_{D^-(S)}(\widetilde{\mathbb{Z}}_S(X)/\widetilde{\mathbb{Z}}_S(U),-[i])\cong\mathbb{H}^i_Z(X,-).\]
\end{proposition}

\begin{proof}
See for example, \cite[Exercise 13.5]{MVW}.
\end{proof}

The following definition comes from \cite[Definition 9.2]{MVW}.
\begin{definition}
Define \(\mathscr{E}_{\mathbb{A}}\) to be the smallest thick subcategory of \(D^-(S)\) (resp. \(D(S)\)) such that
\begin{enumerate}
\item\(Cone(\widetilde{\mathbb{Z}}_S(X\times_k\mathbb{A}^1)\longrightarrow\widetilde{\mathbb{Z}}_S(X))\in\mathscr{E}_{\mathbb{A}}.\)
\item\(\mathscr{E}_{\mathbb{A}}\) is closed under arbitrary direct sums if it exists in \(D^-(S)\) (resp. \(D(S)\)).
\end{enumerate}
Set \(W_{\mathbb{A}}\) to be the class of morphisms in \(D^-(S)\) (resp. \(D(S)\)) whose cone is in \(\mathscr{E}_{\mathbb{A}}\). Define
\[\widetilde{DM}^{eff,-}(S)=D^-(S)[W_{\mathbb{A}}^{-1}]\]
\[\textrm{(resp. }\widetilde{DM}^{eff}(S)=D(S)[W_{\mathbb{A}}^{-1}]\textrm{)}\]
to be the category of effective MW-motives, via Verdier localization (see \cite[4.6]{Kra}). The morphisms in \(D^-(S)\) (resp. \(D(S)\)) becoming isomorphisms after localization by \(W_{\mathbb{A}}\) are called \(\mathbb{A}^1\)-weak equivalences.
\end{definition}
Here, slightly abusing notation, we still denote by $\widetilde{\mathbb{Z}}_S(X)$ the class of $\widetilde{\mathbb{Z}}_S(X)$ (seen as a complex concentrated in degree $0$) in this category. 
\begin{proposition}\label{embedding}
The (naively defined) functor
\[
D^-(S)\to D(S)
\]
induces an exact functor 
\[
\widetilde{DM}^{eff,-}(S)\to \widetilde{DM}^{eff}(S).
\]
which is fully faithful if \(S=pt\).
\end{proposition}

\begin{proof}
The functor \(D^-(S)\longrightarrow D(S)\) is exact by \cite[Proposition 6.7]{Y}. The functor \(\widetilde{DM}^{eff,-}(S)\longrightarrow\widetilde{DM}^{eff}(S)\) is induced and exact by the universal property of the Verdier localization (see \cite[Proposition 4.6.2]{Kra}).

Now suppose \(X,Y\in\widetilde{DM}^{eff,-}(pt)\). We have a commutative diagram
\[
	\xymatrix
	{
		Hom_{\widetilde{DM}^{eff,-}(pt)}(X,Y)\ar[r]^-{\cong}_-u\ar[d]_{\alpha}	&Hom_{\widetilde{DM}^{eff,-}(pt)}(C_*X,C_*Y)\ar[d]	&Hom_{D^-(pt)}(C_*X,C_*Y)\ar[l]_-{\gamma}\ar[d]_{\cong}\\
		Hom_{\widetilde{DM}^{eff}(pt)}(X,Y)\ar[r]	^-{\cong}_-v				&Hom_{\widetilde{DM}^{eff}(pt)}(C_*X,C_*Y)		&Hom_{D(pt)}(C_*X,C_*Y)\ar[l]_-{\beta}
	},
\]
where \(u, v\) are induced by the natural morphisms \(X\longrightarrow C_*X\) and \(Y\longrightarrow C_*Y\) and \(\beta, \gamma\) are isomorphisms by \cite[Lemma 9.19]{MVW} and \cite[Corollary 3.2.11]{DF}. It follows that \(\alpha\) is bijective.
\end{proof}
\begin{definition}
We say that a presheaf with MW-transfers is \emph{free} if it's a direct sum of sheaves of the form \(\widetilde{c}_S(X)\). If a presheaf with MW-transfers is a direct summand of a free presheaf with MW-transfers, we say it's projective. A sheaf with MW-transfers is called free (resp. projective) if it's a sheafication of a free (resp. projective) presheaf with MW-transfers. A bounded above complex of sheaves with MW-transfers is called free (resp. projective) if all its term are free (resp. projective).
\end{definition}
\begin{definition}
A projective resolution of a bounded above complex of sheaves \(K\) is a projective complex (of sheaves) with a quasi-isomorphism \(P\longrightarrow K\).
\end{definition}

In the definition above, if \(K\) is already projective we may take \(P=K\).

\begin{proposition}\label{derived2}
\begin{enumerate}
\item There is a tensor product
\[\begin{array}{cccccc}\otimes_S:&\widetilde{DM}^{eff,-}(S)&\times&\widetilde{DM}^{eff,-}(S)&\longrightarrow&\widetilde{DM}^{eff,-}(S)\\&(K&,&L)&\longmapsto&P\otimes_SQ\end{array},\]
where \(P, Q\) are projective resolutions of \(K, L\) respectively, and $P\otimes_SQ$ is the total complex of the bicomplex \(\{P_i\otimes_SQ_j\}\). Furthermore, for any \(K\in\widetilde{DM}^{eff,-}(S)\), the functor \(K\otimes_S-\) is exact.
\item Suppose that \(f:S\longrightarrow T\) is a smooth morphism in \(Sm/k\). There is an exact functor
\[f_{\#}:\widetilde{DM}^{eff,-}(S)\longrightarrow\widetilde{DM}^{eff,-}(T)\]
defined on objects by $K\mapsto  f_{\#}P$, where \(P\) is a projective resolution of \(K\).
\item Suppose that \(f:S\longrightarrow T\) is a morphism in \(Sm/k\). There is an exact functor
\[f^*:\widetilde{DM}^{eff,-}(T)\longrightarrow\widetilde{DM}^{eff,-}(S)\]
defined on objects by $K\mapsto f^*P$, where \(P\) is a projective resolution of \(K\). Moreover, if \(f\) is smooth, there is an adjunction
\[f_{\#}:\widetilde{DM}^{eff,-}(S)\rightleftharpoons\widetilde{DM}^{eff,-}(T):f^*.\]
\end{enumerate}
\end{proposition}
\begin{proof}
See \cite[Proposition 6.12 and Proposition 6.13]{Y}.
\end{proof}
\begin{coro}\label{operations1}
Let \(f:S\longrightarrow T\) be a morphism in \(Sm/k\).
\begin{enumerate}
\item For any \(K, L\in\widetilde{DM}^{eff,-}(T)\), we have
\[f^*(K\otimes_TL)\cong(f^*K)\otimes_S(f^*L).\]
\item If \(f\) is smooth, then for any \(K\in\widetilde{DM}^{eff,-}(S)\) and \(L\in\widetilde{DM}^{eff,-}(T)\), we have
\[f_{\#}(K\otimes_Sf^*L)\cong(f_{\#}K)\otimes_TL.\]
\end{enumerate}
\end{coro}
By \cite[Example 3.15]{CD}, the categories \(\widetilde{DM}^{eff}(S)\) also have operations \(\otimes,f^*,f_{\#}\), where projective resolutions are replaced by cofibrant resolutions. These operations have the same properties as above. By \cite[Proposition 6.6]{Y}, \(\otimes,f^*,f_{\#}\) are compatible (up to a natural isomorphism) with the embedding functor \(\widetilde{DM}^{eff,-}(S)\to \widetilde{DM}^{eff}(S)\) defined in Proposition \ref{embedding}.
\begin{proposition}\label{zarsep}
Let \(X\in Sm/k\), \(\{U_i\}\) be an open covering of \(X\) and \(f\) is a morphism in {\(\widetilde{DM}^{eff,-}(X)\)}. If \(f|_{U_i}\) is an isomorphism for every \(i\), then \(f\) is an isomorphism.
\end{proposition}
\begin{proof}
We may assume the index \(i\) is finite. By the condition, \(Cone(f)|_{U_i}=0\) in \(\widetilde{DM}^{eff,-}(U_i)\) for every \(i\). Set \(C=Cone(f)\), we are going to prove that \(C=0\) in \(\widetilde{DM}^{eff,-}(X)\). For any \(Y\in Sm/U_i\) and \(n\in\mathbb{Z}\), we have
\[Hom_{\widetilde{DM}^{eff,-}(U_i)}(\widetilde{\mathbb{Z}}_{U_i}(Y)[n],C|_{U_i})=Hom_{\widetilde{DM}^{eff,-}(X)}(\widetilde{\mathbb{Z}}_X(Y)[n],C)=0\]
by adjuction. {Hence by a Mayer-Vietoris argument, we see that the equation above holds for any \(Y\in Sm/X\). Hence \(C=0\) in \(\widetilde{DM}^{eff,-}(X)\) by \cite[Lemma 9.4]{MVW}.}
\end{proof}
\begin{definition}
We say that a morphism \(p:E\longrightarrow X\) in \(Sm/S\) is an \(\mathbb{A}^n\)-bundle if there is an open covering \(\{U_i\}\) of \(X\) such that \(p^{-1}(U_i)\cong U_i\times_k\mathbb{A}^n\).
\end{definition}
\begin{proposition}\label{homotopy invariance}
Let \(p:E\longrightarrow X\) be an \(\mathbb{A}^n\)-bundle and \(X\in Sm/S\). Then, the map 
\[\widetilde{\mathbb{Z}}_S(p):\widetilde{\mathbb{Z}}_S(E)\longrightarrow\widetilde{\mathbb{Z}}_S(X)\] 
is an isomorphism in \(\widetilde{DM}^{eff,-}(S)\).
\end{proposition}
\begin{proof}
Follows by \(\mathbb{A}^1\)-invariance and Proposition \ref{MV}.
\end{proof}

\begin{proposition}\label{deformation to the normal bundle}
(Homotopy Purity) Let \(X\in Sm/S\) and \(Y\subseteq X\) be a smooth closed subscheme. Then
\[\widetilde{M}_Y(X)\cong Th_S(N_{Y/X})\]
in \(\widetilde{DM}^{eff,-}(S)\).
\end{proposition}
\begin{proof}See \cite[Theorem 2.2.8]{P}. Alternatively, one may use \cite[\S 3, Theorem 2.23]{MV} and the sequence of functors of \cite[\S 3.2.4.a]{DF}.
\end{proof}
\begin{proposition}\label{ext}
Suppose \(X_i\in Sm/S\) and \(U_i=X_i\setminus Z_i (i=1,2)\) are open in \(X_i\). Then we have the exterior product
\[\times:\widetilde{M}_{Z_1}(X_1)\otimes_S\widetilde{M}_{Z_2}(X_2)\cong\widetilde{M}_{Z_1\times_S Z_2}(X_1\times_S X_2)\]
in \(\widetilde{DM}^{eff,-}(S)\).
\end{proposition}
\begin{proof}
By Proposition \ref{MV}, the following morphism between complexes is a quasi-isomorphism
\[
	\xymatrix
	{
		\widetilde{\mathbb{Z}}_S(U_1\times_S U_2)\ar[r]\ar[d]	&\widetilde{\mathbb{Z}}_S(U_1\times_S X_2)\oplus{\widetilde{\mathbb{Z}}_S}(U_2\times_S X_1)\ar[r]\ar[d]	&\widetilde{\mathbb{Z}}_S(X_1\times_S X_2)\ar@{=}[d]\\
		0\ar[r]							&\widetilde{\mathbb{Z}}_S((Z_1\times_S Z_2)^c)\ar[r]								&\widetilde{\mathbb{Z}}_S(X_1\times_S X_2)
	},
\]
where the first row (resp. second row) is the left hand side (resp. right hand side) of the statement.
\end{proof}
Now in the notation above, suppose \(X_1=X_2=X\). If we have two maps \(f_i:\widetilde{M}_{Z_i}(X)\longrightarrow C_i\), \(i=1,2\) in \(\widetilde{DM}^{eff,-}(S)\) we define \(f_1\boxtimes f_2\) as the composite
\[\xymatrix{\widetilde{M}_{Z_1\cap Z_2}(X)\ar[r]^-{\triangle}&\widetilde{M}_{Z_1\times_S Z_2}(X\times_S X)\ar[r]^-{\times}&\widetilde{M}_{Z_1}(X)\otimes_S\widetilde{M}_{Z_2}(X)\ar[r]^-{f_1\otimes f_2}&C_1\otimes_S C_2}.\]
\begin{coro}\label{additivity}
\begin{enumerate}
\item Suppose that \(f:S\longrightarrow T\) is a morphism in \(Sm/k\), that \(X\in Sm/T\) and that \(E\) is a vector bundle over \(X\). Then we have
\[f^*Th_T(E)\cong Th_S(f^*E)\]
in \(\widetilde{DM}^{eff,-}(S)\), where \(f^*E\) the pull-back of \(E\) from \(X\) to \(X\times_TS\) via \(f\).
\item Suppose that \(f:S\longrightarrow T\) is a smooth morphism in \(Sm/k\), that \(X\in Sm/S\) and that \(E\) is a vector bundle over \(X\). Then we have
\[f_{\#}Th_S(E)\cong Th_T(E)\]
in \(\widetilde{DM}^{eff,-}( T)\).
\item (\cite[Remark 2.4.15]{CD1}) Suppose \(E_1\) and \(E_2\) are vector bundles over \(X\in Sm/k\). Then
\[Th_X(E_1)\otimes_XTh_X(E_2)\cong Th_X(E_1\oplus E_2)\]
in \(\widetilde{DM}^{eff,-}(X)\).
\end{enumerate}
\end{coro}
\begin{proposition}\label{punctured affine space}
For any $n\in\mathbb{N}$ and \(S\in Sm/k\), we have an isomorphism
\[t_n:\widetilde{M}_{0\times S}(\mathbb{A}^{n}\times S)\cong\widetilde{\mathbb{Z}}_S(n)[2n]\]
in \(\widetilde{DM}^{eff,-}(S)\), where \(t_n=t_1^{\times n}\).
\end{proposition}
\begin{proof}The statement is trivial if \(n=0, 1\). Now for any \(n,m\in\mathbb{N}\), we have
\[\widetilde{M}_{{0\times S}}(\mathbb{A}^{n}\times S)\otimes_S\widetilde{M}_{{0\times S}}(\mathbb{A}^{m}\times S)\cong\widetilde{M}_{{0\times S}}(\mathbb{A}^{n+m}\times S)\]
by Proposition \ref{ext}.
\end{proof}
\begin{definition}\label{CPW}
For any \(X\in Sm/k\), \(Z\) closed in \(X\), \(i\in\mathbb{N}\) and \(j\in\mathbb{Z}\), we define
\[A_Z^{j,i}(X)=Hom_{\widetilde{DM}^{eff,-}(pt)}(\widetilde{M}_Z(X),\widetilde{\mathbb{Z}}_{pt}(i)[j])\]
where we define \(A^{j,i}(X)=A_X^{j,i}(X)\). Then for any \(Z_1, Z_2\) closed in \(X\), \(i_1,i_2\in\mathbb{N}\) and \(j_1,j_2\in\mathbb{Z}\), we have a product
{\[\xymatrix{A_{Z_1}^{j_1,i_1}(X)\times A_{Z_2}^{j_2,i_2}(X)\ar[r]^-{\boxtimes}&A_{Z_1\cap Z_2}^{j_1+j_2,i_1+i_2}(X)},\]}
which satisfies the axioms given in \cite[Definition 2.1 and Definition 2.2]{PW} by Proposition \ref{excision}.

\end{definition}
\begin{proposition}\label{base}
Let \(X\) be a smooth scheme, \(Z\subseteq X\) be a closed subset and \(i\geq 0\). Then
\[\mathbb{H}^{2i}_Z(X,\widetilde{\mathbb{Z}}_{pt}(i))\cong\widetilde{CH}^i_Z(X),\]
and in particular
\[\mathbb{H}^{2i}(X,\widetilde{\mathbb{Z}}_{pt}(i))\cong\widetilde{CH}^i(X)\]
functorially in $X$.
Moreover, the following diagram commutes for any \(i,j\geq 0\)
\[
	\xymatrix@C=1.0em
	{
		A^{2i,i}(X)\times A^{2j,j}(X)\ar[r]\ar[d]^{\boxtimes}	&\widetilde{CH}^i(X)\times\widetilde{CH}^j(X)\ar[d]^{\cdot}\\
		A^{2(i+j),i+j}(X)\ar[r]															&\widetilde{CH}^{i+j}(X)
	}
\]
where the right-hand map is the intersection product on Chow-Witt groups.
Consequently, we have isomorphisms \(A^{2i,i}(X)\longrightarrow\widetilde{CH}^i(X)\) which send the natural embedding \(j:\widetilde{\mathbb{Z}}_{pt}(pt)\longrightarrow\widetilde{\mathbb{Z}}_{pt}\) to \(1\) when \(i=0\) and \(X=pt\).
\end{proposition}
\begin{proof}
See \cite[Corollary 4.2.6]{DF}.
\end{proof}
\begin{proposition}\label{cancellation}
Let \(K, L\in\widetilde{DM}^{eff,-}(pt)\). The map
\[\xymatrix{Hom_{\widetilde{DM}^{eff,-}(pt)}(K,L)\ar[r]^-{\otimes_{pt}\widetilde{\mathbb{Z}}_{pt}(i)}&Hom_{\widetilde{DM}^{eff,-}(pt)}(K(i),L(i))}, i>0\]
is an isomorphism.
\end{proposition}
\begin{proof}See \cite[Theorem 5.0.1]{FO} by using Proposition \ref{embedding}.
\end{proof}
\begin{lemma}\label{orthogonality}
Let \(X\in Sm/k\) and let \(i,j\geq 0\). Then
\[
Hom_{\widetilde{DM}^{eff,-}(pt)}(\widetilde{\mathbb{Z}}_{pt}(X)(i)[2i],\widetilde{\mathbb{Z}}_{pt}(j)[2j])=\begin{cases} 0 & \text{ if $i>j$.} \\ \widetilde{CH}^{j-i}(X) & \text{ if $i\leq j$.} \end{cases} 
\]
\end{lemma}

\begin{proof}
If \(i\leq j\), the lemma follows from Propositions \ref{base} and  \ref{cancellation}. Suppose then that $i>j$. Tensoring the isomorphism in Proposition \ref{punctured affine space} with \(\widetilde{\mathbb{Z}}_{pt}(X)\), we get 
\[
\widetilde{\mathbb{Z}}_{pt}(X)(i)[2i]\simeq \widetilde{\mathbb{Z}}_{pt}(X\times \mathbb{A}^i)/\widetilde{\mathbb{Z}}_{pt}(X\times (\mathbb{A}^i\setminus 0)).
\]
Then it follows from Propositions \ref{prop:relative} and \ref{base} that 
\[
Hom_{\widetilde{DM}^{eff,-}(pt)}(\widetilde{\mathbb{Z}}_{pt}(X)(i)[2i],\widetilde{\mathbb{Z}}_{pt}(j)[2j])\simeq \widetilde{CH}^{j}_{X\times 0}(X\times \mathbb{A}^i)=0.
\]
\end{proof}

\begin{coro}\label{cor:orthogonality}
For any $i,j\geq 0$, we have 
\[
Hom_{\widetilde{DM}^{eff,-}(pt)}(\widetilde{\mathbb{Z}}_{pt}(i)[2i],\widetilde{\mathbb{Z}}_{pt}(j)[2j])=\begin{cases} 0 & \text{ if $i\neq j$.} \\ \widetilde{CH}^{0}(k) & \text{ if $i= j$.} \end{cases} 
\]
In other terms, the motives $\widetilde{\mathbb{Z}}_{pt}(i)[2i]$ are mutually orthogonal in the triangulated category $\widetilde{DM}^{eff,-}(pt)$.
\end{coro}

\section{Quaternionic geometry}\label{Grass}
\subsection{Grassmannian bundles and quaternionic projective bundles}
First of all, we recall the basics on Grassmannian bundles and quaternionic projective bundles. Although these are well-known objects, we include the definitions here for the sake of notations. The reader may refer to \cite{KL}, \cite{S} for Grassmannians, \cite{K} for Grassmannian bundles and \cite{PW} for quaternionic projective bundles.
\begin{definition}
Let \(X\) be a \(S\)-scheme, \(\mathscr{E}\) locally free of rank \(r\) on \(X\), \(1\leq n\leq r\). Define a functor
\[
	\begin{array}{ccc}
		F:{X-Sch}^{op}		&\longrightarrow	&Set\\
		f:T\longrightarrow X	&\longmapsto		&\{\mathscr{F}\subseteq f^{*}\mathscr{E}|f^{*}\mathscr{E}/\mathscr{F}\textrm{ is locally free of rank n}\}
	\end{array}
\]
with functorial maps defined by pull-backs. If \(F\) is representable, the representative is called the Grassmannian bundle of rank \(n\) of \(\mathscr{E}\), denoted by \(Gr_{X}(n,\mathscr{E})\).
\end{definition}
\begin{proposition}
The functor \(F\) is representable. Further, if \(\mathscr{E}\cong O_{X}^{\oplus r}\), then \(Gr_{X}(n,\mathscr{E})\cong Gr(n,r)\times_{k}X\) over \(X\), where \(Gr(n,r)\) is the Grassmannian of rank \(r\) of \(k^{\oplus n}\).
\end{proposition}
\begin{proof}See \cite[Proposition 1.2]{K}.
\end{proof}

Let \(p:Gr_{X}(n,\mathscr{E})\longrightarrow X\) be the structure map. There is a universal element \(\mathscr{F}\subseteq p^{*}\mathscr{E}\) with quotient of rank $n$. The vector bundle \((p^{*}\mathscr{E}/\mathscr{F})^{\vee}\) is called the tautological bundle of \(Gr_X(n,\mathscr{E})\), denoted by \(\mathscr{U}\).  Its dual is just called the dual tautological bundle, denoted by \(\mathscr{U}^{\vee}\).
\begin{definition}
Let \(\mathscr{E}\neq 0\) be a locally free sheaf of rank \(n\) over a scheme \(X\). It is called symplectic if one equips it with a skew-symmetric \((v\cdot v=0)\) and non degenerate inner product \(m:\mathscr{E}\times\mathscr{E}\longrightarrow O_X\) (hence \(n\) is always even). 
\end{definition}

Now let \(f:X\longrightarrow Y\) be a morphism of schemes and \((\mathscr{E},m)\) be a symplectic bundle on \(Y\). Then \((f^{*}\mathscr{E},f^{*}(m))\) is also a symplectic bundle, where \(f^{*}(m)\) is the pull back of the map \(\mathscr{E}\longrightarrow\mathscr{E}^{\vee}\) induced by \(m\).

The following is a basic tool when dealing with non degeneracy of inner products.
\begin{proposition}\label{non degeneracy}
Let \(f:X\longrightarrow Y\) be a morphism between schemes and \(\mathscr{E}\) be a locally free sheaf of finite rank over \(Y\) with an inner product \(m:\mathscr{E}\times\mathscr{E}\longrightarrow O_X\). Then for any \(x\in X\), \(m\) is non degenerate at \(f(x)\) if and only if \(f^*(m)\) is non degenerate at \(x\).
\end{proposition}
\begin{proof}
This is basically because \(f\) induces local homomorphisms between stalks.
\end{proof}
\begin{proposition}\label{prop:symplectic}
Suppose we have an injection \(i:\mathscr{E}_{1}\longrightarrow\mathscr{E}_{2}\) between vector bundles, where \(\mathscr{E}_{2}\) admits a nondenerate inner product \(m\) and \({m_{\mathscr{E}_{2}}}|_{\mathscr{E}_{1}}\) is non degenerate. Define \(\mathscr{E}_{1}^{\perp}(U):={\mathscr{E}_{1}(U)}^{\perp}\) for every \(U\). Then \(\mathscr{E}_{1}^{\perp}\) is again a vector bundle with a nondegenerate inner product inherited from \(\mathscr{E}_{2}\) and there exists a unique \(p:\mathscr{E}_{2}\longrightarrow\mathscr{E}_{1}\) with \(p\circ i=id_{\mathscr{E}_{1}}\) and \(Im(id_{\mathscr{E}_{2}}-i\circ p)\subseteq\mathscr{E}_{1}^{\perp}\). In other words, \(\mathscr{E}_2\cong\mathscr{E}_1\oplus\mathscr{E}_1^{\perp}\).
\end{proposition}
\begin{proof}
We define \(p\) by the commuative diagram
\[
	\xymatrix
	{
		\mathscr{E}_2^{\vee}\ar[r]^{i^{\vee}}	&\mathscr{E}_1^{\vee}\\
		\mathscr{E}_2\ar[u]_{\cong}\ar[r]^{p}	&\mathscr{E}_1\ar[u]_{\cong}
	},
\]
where vertical maps are induced by inner products. It's easy to check this is what we want.
\end{proof}
Now suppose in this section \((\mathscr{E},m)\) is a symplectic bundle of rank \(2n+2\) over a scheme \(X\).

\begin{definition}\label{quaternionic projective bundle}
Define a functor
\[
	\begin{array}{ccc}
		H:{X-Sch}^{op}		&\longrightarrow	&Set\\
		f:T\longrightarrow X	&\longmapsto		&\{\mathscr{F}\subseteq f^{*}\mathscr{E}|f^{*}(m)|_{\mathscr{F}}\textrm{ non degenerate}, f^{*}\mathscr{E}/\mathscr{F}\textrm{ v.b. of rank }2n\}
	\end{array}
\]
with functorial maps defined by pull-backs. 
\end{definition}
\begin{definition}\label{HPn}
Let
\[HP^{n}=D(\sum_{i=1}^{n+1}p_{i,i+n+1})\subseteq Gr(2,2n+2),\]
where \(\{p_{i,j}\}\) are Pl\"ucker coordinates of \(Gr(2,2n+2)\). It's just the set of two dimensional subspaces of \(k^{\oplus 2n+2}\) on which the standard symplectic form \(\left(\begin{array}{cc}&I\\-I&\end{array}\right)\) is non-degenerate.
\end{definition}

\begin{proposition}\label{HGr}
The functor \(H\) is representable by a scheme \(HGr_{X}(\mathscr{E})\). Further, if \((\mathscr{E},m)\cong\left(O_{X}^{\oplus 2n+2},\left(\begin{array}{cc}&I\\-I&\end{array}\right)\right)\), then \(HGr_{X}(\mathscr{E})\cong HP^{n}\times_{k}X\) over \(X\).
\end{proposition}
\begin{proof}The result is standard. We have the structure map \(\pi:Gr_X(2n,\mathscr{E})\longrightarrow X\) and the tautological exact sequence
\[0\longrightarrow\mathscr{F}\longrightarrow\pi^*\mathscr{E}\longrightarrow\mathscr{U}^{\vee}\longrightarrow 0.\]
Define
\[HGr_X(\mathscr{E})=\{x\in Gr_X(2n,\mathscr{E})|\pi^*(m)|_{\mathscr{F}}\textrm{ is non degenerate at }x\}.\]
\end{proof}

\begin{definition}
We will call \(HGr_{X}(\mathscr{E})\) the quaternionic projective bundle of $\mathscr{E}$. 
\end{definition}
Note that any morphism $f:Y\to X$ between schemes yields a commutative diagram
\[
\xymatrix{HGr_Y(f^*\mathscr{E})\ar[r]^{f^*}\ar[d] & HGr_X(\mathscr{E})\ar[d] \\
Y\ar[r]_-f & X}.
\]

Let \(p:HGr_{X}(\mathscr{E})\longrightarrow X\) be the structure map. Then, there is a universal element \(\mathscr{F}\subseteq p^{*}\mathscr{E}\) which is just obtained by the restriction of the universal element of the Grassmannian bundle to \(HGr_X(\mathscr{E})\). The vector bundle \(\mathscr{F}\) itself is called the tautological bundle of \(HGr_X(\mathscr{E})\), denoted by \(\mathscr{U}\).  Its dual is just called the dual tautological bundle, denoted by \(\mathscr{U}^{\vee}\). We will use the same symbol \(\mathscr{U}\) for all tautological bundles defined above if there is no confusion. Note that both $\mathscr U$ and $\mathscr U^\vee$ are symplectic by Proposition \ref{prop:symplectic}. 
\subsection{Symplectic Thom structure}
Let \(X\in Sm/S\) and \((\mathscr{E},m)\) be a symplectic vector bundle of rank \(2n\) over \(X\) with total space \(E\).

Recall that, as in the discussion before \cite[Theorem 4.1]{PW}, \(O_X\oplus\mathscr{E}\oplus O_X\) is also a symplectic vector bundle with inner product \(\left(\begin{array}{ccc}0&0&1\\0&m&0\\-1&0&0\end{array}\right)\).
\begin{definition}
\begin{enumerate}
\item Define \(N^-\) by the cartesian square
\[
	\xymatrix
	{
		Gr_X(2n,\mathscr{E}\oplus O_X)\ar[r]^-{i}	&Gr_X(2n,O_X\oplus\mathscr{E}\oplus O_X)\\
		N^-\ar[u]\ar[r]					&HGr_X(O_X\oplus\mathscr{E}\oplus O_X)\ar[u]_{j}
	},
\]
where \(i\) comes from the projection \(p_{23}:O_X\oplus\mathscr{E}\oplus O_X\longrightarrow\mathscr{E}\oplus O_X\) and \(j\) is the inclusion (see Proposition \ref{HGr}).
\item Define
\[N=\{x\in Gr_X(2n,O_X\oplus\mathscr{E}\oplus O_X)|\mathscr{E}'\longrightarrow p^*(O_X\oplus\mathscr{E}\oplus O_X)\longrightarrow p^*(O_X\oplus O_X)\textrm{ iso. at x}\},\]
where \(p:Gr_X(2n,O_X\oplus\mathscr{E}\oplus O_X)\longrightarrow X\) is the structure map and 
\[0\longrightarrow\mathscr{E}'\longrightarrow p^*(O_X\oplus\mathscr{E}\oplus O_X)\longrightarrow\mathscr{E}''\longrightarrow 0\]
is the tautological exact sequence. 
Note that \(N\) is an open set of the Grassmannian \(Gr_X(2n,O_X\oplus\mathscr{E}\oplus O_X)\).
\item Define
\[V=\{x\in Gr_X(2n,\mathscr{E}\oplus O_X)|\mathscr{F}'\longrightarrow q^*(\mathscr{E}\oplus O_X)\longrightarrow q^*O_X\textrm{ is an isomorphism at x}\},\]
where \(q:Gr_X(2n,\mathscr{E}\oplus O_X)\longrightarrow X\) is the structure map and 
\[0\longrightarrow\mathscr{F}'\longrightarrow q^*(\mathscr{E}\oplus O_X)\longrightarrow\mathscr{F}''\longrightarrow 0\]
is the tautological exact sequence. As above, note that \(V\) is an open set of \(Gr_X(2n,\mathscr{E}\oplus O_X)\).
\end{enumerate}
\end{definition}

The notations of \(N^-\) and \(N\) come from \cite[Theorem 4.1]{PW}, but our treatment is slightly different.
\begin{lemma}\label{open condition}
1) Let \(T\) be an \(X\)-scheme and \(f:T\longrightarrow Gr_X(2n,O_X\oplus\mathscr{E}\oplus O_X)\) be an \(X\)-morphism. Then
\[Im(f)\subseteq N\Longleftrightarrow f^*\mathscr{E}'\longrightarrow (p\circ f)^*(O_X\oplus\mathscr{E}\oplus O_X)\longrightarrow (p\circ f)^*(O_X\oplus O_X)\textrm{ is an isomorphism}.\]
Consequently, \(N^-\subseteq N\cap HGr_X(O_X\oplus\mathscr{E}\oplus O_X)\).\\
2) Let \(T\) be an \(X\)-scheme and \(f:T\longrightarrow Gr_X(2n,\mathscr{E}\oplus O_X)\) be an \(X\)-morphism. Then
\[Im(f)\subseteq V\Longleftrightarrow f^*\mathscr{F}'\longrightarrow (q\circ f)^*(\mathscr{E}\oplus O_X)\longrightarrow (q\circ f)^*O_X\textrm{ is an isomorphism}.\]
Furthermore, \(N^-=V\).
\end{lemma}
\begin{proof}1) \(\Longrightarrow\) Easy. For the \(\Longleftarrow\) part, set
\[C=Coker(\mathscr{E}'\longrightarrow p^*(O_X\oplus\mathscr{E}\oplus O_X)\longrightarrow p^*(O_X\oplus O_X)).\]
We see that \(N=Supp(C)^c\). Since \(f^{-1}(Supp(C))=Supp(f^*C)\), \(f^{-1}(Supp(C))=\emptyset\) hence \(f^{-1}(N)=T\). So \(Im(f)\subseteq N\).

{For the second statement, let \(v:N^-\longrightarrow X\) be the structure map. We have a commutative diagram with exact rows:
\[
	\xymatrix
	{
		0\ar[r]	&K'\ar[r]					&v^*(\mathscr{E}\oplus O_X)\ar[r]					&\mathscr{G}\ar[r]			&0\\
		0\ar[r]	&O_X\oplus K'\ar[r]\ar[u]^{p_2}	&v^*(O_X\oplus\mathscr{E}\oplus O_X)\ar[r]\ar[u]^{p_{23}}	&\mathscr{G}\ar[r]\ar@{=}[u]	&0
	},
\]
where the first row is the pull-back of the tautological exact sequence of \(Gr_X(2n,\mathscr{E}\oplus O_X)\). So for every \(x\in N^-\), there is an affine neighborhood \(U\) of \(x\) such that \(K(U)\) is a free \(O_{N^-}(U)\)-module with a basis \((1,0,0)\) and \((0,x_1,x_2)\). Hence \(x_2\in O_{N^-}(U)^*\) by non degeneracy. Hence the map \(O_X\oplus K'\longrightarrow v^*(O_X\oplus\mathscr{E}\oplus O_X)\longrightarrow v^*(O_X\oplus O_X)\) is surjective on \(U\). So we see that \(N^-\subseteq N\) by the first statement.}

2) The first statement can be proved in the same way as in 1). {For the second statement, 
we see from the discussion above that the composite \(K'\longrightarrow v^*(\mathscr{E}\oplus O_X)\longrightarrow v^*O_X\) is an isomorphism.} So \(N^-\subseteq V\). The inclusion \(V\subseteq N^-\) can be proved using a similar method.
\end{proof}
\begin{lemma}\label{open condition1}
Let \(T\) be an \(X\)-scheme and \(f:T\longrightarrow Gr_X(2n,O_X\oplus\mathscr{E}\oplus O_X)\) be an \(X\)-morphism. Let \(\varphi\) be the composite 
\[
\xymatrix{(p\circ f)^*O_X\ar[r]^-{i_1}&(p\circ f)^*(O_X\oplus\mathscr{E}\oplus O_X)\ar[r]&f^*\mathscr{E}''}.
\] 
Then
\[Im(f)\subseteq Gr_X(2n,\mathscr{E}\oplus O_X)^c\Longleftrightarrow\varphi\textrm{ is injective and has a locally free cokernel}.\]
\end{lemma}
\begin{proof}
\[Im(f)\subseteq Gr_X(2n,\mathscr{E}\oplus O_X)^c\Longleftrightarrow\forall g:\textrm{Spec }K\longrightarrow T, Im(f\circ g)\subseteq Gr_X(2n,\mathscr{E}\oplus O_X)^c,\]
where \(K\) is a field. So let's assume \(T=\textrm{Spec }K\). In this case,
\[Im(f)\subseteq Gr_X(2n,\mathscr{E}\oplus O_X)^c\Longleftrightarrow f\textrm{ does not factor through }Gr_X(2n,\mathscr{E}\oplus O_X),\]
and the latter condition is equivalent to \(\varphi\neq 0\). Hence
\[Im(f)\subseteq Gr_X(2n,\mathscr{E}\oplus O_X)^c\Longleftrightarrow\forall g:\textrm{Spec }K\longrightarrow T, g^*(\varphi)\neq 0.\]
Now we may assume that \(T\) is affine and use the residue fields of \(T\). Locally, the map \(\varphi\) is like \((a_i):A\longrightarrow A^{\oplus 2n}\) and the condition just says that the ideal \((a_i)\) is the unit ideal, which is equivalent to \((a_i)\) being injective and \(Coker((a_i))\) being projective. This just says that \(\varphi\) is injective and has a locally free cokernel.
\end{proof}

Consider next the following square
\[
	\xymatrix
	{
		N^-\ar[r]^-l\ar[d]_-v	&N\ar[d]^-u\\
		X\ar[r]_-z			&E
	}
\]
where \(l\) is given by \(N^-\subseteq N\) and \(v\) is just the structure map (of \(N^-\)). Let \(r:N\longrightarrow X\) be the structure map of \(N\). We have the tautological exact sequence
\[
	\begin{array}{cccccccccc}
		0	&\longrightarrow	&r^*(O_X\oplus O_X)	&\longrightarrow	&r^*(O_X\oplus\mathscr{E}\oplus O_X)	&\longrightarrow	&r^*\mathscr{E}	&\longrightarrow	&0	&\\
			&				&(1,0)			&\longmapsto		&(1,s_1,0)						&				&			&				&	&(**)\\
			&				&(0,1)			&\longmapsto		&(0,s_2,1)						&				&			&				&	&
	\end{array}
\]
and \(u\) is induced by \(s_1\). Finally, \(z\) is the zero section of \(E\).
\begin{proposition}\label{Cartesian}
The above square is a Cartesian square.
\end{proposition}
\begin{proof}The map \(l\) induces an exact sequence
\[
	\begin{array}{ccccccccc}
		0	&\longrightarrow	&v^*(O_X\oplus O_X)	&\longrightarrow	&v^*(O_X\oplus\mathscr{E}\oplus O_X)	&\longrightarrow	&v^*\mathscr{E}	&\longrightarrow	&0\\
			&				&(1,0)			&\longmapsto		&(1,s,0)						&				&			&				&
	\end{array}.
\]
But \((1,0,0)\) belongs to the kernel since \(N^-\subseteq Gr_X(2n,\mathscr{E}\oplus O_X)\), so \(s=0\). Hence the square commutes and is Cartesian by \(N^-=V\).
\end{proof}

Now, we use the square
\[
	\xymatrix
	{
		N\ar[r]^-w\ar[d]_-u	&E\ar[d]^-{\pi}\\
		E\ar[r]_-{\pi}		&X,
	}
\]
where \(w\) is induced by \(s_2\) in (**). We see that it's a Cartesian square just by that diagram, since there are two (arbitrary) sections there. So \(u\) is an \(\mathbb{A}^{2n}\)-bundle.

The following proposition has a similar version in \cite[Proposition 4.3]{PW}, but we are not considering the same embedding as there.
\begin{proposition}\label{thomdefor}
\[Th_S(E)\cong\widetilde{M}_{N^-}(N)\cong\widetilde{M}_{N^-}(HGr_X(O_X\oplus\mathscr{E}\oplus O_X))\]
in \(\widetilde{DM}^{eff,-}(S)\).
\end{proposition}
\begin{proof}
The first isomorphism comes from Proposition \ref{Cartesian} and the fact that \(u:N\longrightarrow E\) is an \(\mathbb{A}^{2n}\)-bundle. Then second isomorphism is because \(N^-\subseteq N\cap HGr_X(O_X\oplus\mathscr{E}\oplus O_X)\) by Lemma \ref{open condition} and Proposition \ref{excision}.
\end{proof}
Now by Lemma \ref{open condition1}, the natural embedding \(HGr_X(\mathscr{E})\longrightarrow HGr_X(O_X\oplus\mathscr{E}\oplus O_X)\) factor through \((N^-)^c\), thus we have a map \(i:HGr_X(\mathscr{E})\longrightarrow (N^-)^c\).
\begin{proposition}\label{deformation}
\[\widetilde{\mathbb{Z}}_S(i):\widetilde{\mathbb{Z}}_S(HGr_X(\mathscr{E}))\longrightarrow\widetilde{\mathbb{Z}}_S((N^-)^c)\]
is an isomorphism in \(\widetilde{DM}^{eff,-}(S)\).
\end{proposition}
\begin{proof}Follows from the proof of \cite[Theorem 5.2]{PW}.
\end{proof}

\section{The \(SL\)-orientation of Chow-Witt rings}\label{sl}

Now let's discuss the notion of \(SL\)-orientation.
\begin{definition}
Let \(X\) be a scheme and let \(\mathscr{E}\) be a vector bundle over \(X\). A section \(s\in(det\mathscr{E}^{\vee})(X)\) is called an \(SL\)-orientation of \(\mathscr{E}\) if \(s\) trivializes \(\textrm{det}\mathscr{E}^{\vee}\). A vector bundle with an \(SL\)-orientation is called \(SL\)-orientable.

Let \(\mathscr{E}_1\),\(\mathscr{E}_2\) be two \(SL\)-orientable vector bundles over a scheme \(X\) with \(SL\)-orientations \(s_1\), \(s_2\), respectively. A morphism \(f:\mathscr{E}_1\longrightarrow\mathscr{E}_2\) between \(SL\)-orientable bundles is an isomorphism between vector bundles satisfiying \(det(f)^{\vee}(s_2)=s_1\).
\end{definition}
\begin{definition}

{Let \((\mathscr{E},m)\) be a symplectic bundle. Then \(\mathscr{E}\) has an \(SL\)-orientation induced by \(m\) by remarks before \cite[Definition 4.5]{A3}, which is called its canonical orientation.}
\end{definition}

In the case \(rk\mathscr{E}=2\), there is a bijection
\[\xymatrixcolsep{6pc}\xymatrix{\{\textrm{Symplectic structures on \(\mathscr{E}\)}\}\ar[r]^-{\textrm{canonical orientation}}&\{\textrm{\(SL\)-Orientations of \(\mathscr{E}\)}\}}.\]
So for convenience, if \(\mathscr{E}\) is symplectic, {\(\mathscr{E}\) is always endowed with its canonical orientation.}

The following proposition from \cite[Corollary 5.4]{A2} says that Chow-Witt rings theory admits an (unique) \(SL\)-orientation.
\begin{proposition}\label{SL}
Suppose \(X\in Sm/k\), \(E\) is an \(SL\)-orientable bundle of rank \(n\) over \(X\), \(i\geq n\) and \(j\in\mathbb{Z}\). We have a functorial element
\[th(E)\in A_X^{2n,n}(E)\]
such that the morphism
\[-\boxtimes th(E):A^{j-2n,i-n}(X)\longrightarrow A_X^{j,i}(E)\]
is an isomorphism.

Moreover, \(th(E)\) is determined by the following property: For any \(U\) open in \(X\) such that \(\xymatrix{E|_U\ar[r]^-{\alpha}_-{\cong}& O_U^{\oplus rk(E)}}\) as \(SL\)-orientable bundles, we have
\[th(E)|_U=t_n\circ Th_{pt}(\alpha).\]
\end{proposition}
From the proposition above, we see that {\(A_Z^{j,i}(X)\)} in Definition \ref{CPW} satisfies \cite[Definition 7.1]{PW}.

The following definition comes from \cite[remark after Proposition 7.2]{PW} and \cite[remark before Proposition 3.1.1]{AF}.
\begin{definition}\label{first Pontryagin class}
Let \(X\in Sm/k\) and \(\mathscr{E}\) be an \(SL\)-orientable vector bundle of rank \(n\) over \(X\) with the zero section \(z\). Define its Euler class \(e(\mathscr{E})\) to be the composite
\[\xymatrixcolsep{3pc}\xymatrix{\widetilde{\mathbb{Z}}_{pt}(X)\ar[r]^z&\widetilde{\mathbb{Z}}_{pt}(E)\ar[r]&Th_{pt}(E)\ar[r]^-{th(E)}&\widetilde{\mathbb{Z}}_{pt}(n)[2n]}.\]
If \(n=2\), define the first Pontryagin class of \(\mathscr{E}\) to be \(-e(\mathscr{E})\), which is denoted by \(p_1(\mathscr{E})\).

{Moreover, suppose \(S\in Sm/k\) and \(X\in Sm/S\), then the Euler class (resp. first Pontryagin class) above induces an element in \(Hom_{\widetilde{DM}^{eff,-}(S)}(\widetilde{\mathbb{Z}}_S(X),\widetilde{\mathbb{Z}}_S(n)[2n])\) by adjunction. This is defined to be the Euler class (resp. first Pontryagin class) of \(\mathscr{E}\) over \(S\).}
\end{definition}
If two \(SL\)-orientable bundles are isomorphic, their Euler classes are the same. In particular, if two symplectic bundles of rank \(2\) are isomorphic (including their inner products) then their first Pontryagin classes under the canonical orientations are equal.

\section{Quaternionic projective bundle theorem}\label{Projbundle}

Now let's start to calculate the motive of \(HP^n\). Let \(x_1,\ldots,x_{2n+2}\) be the coordinates of the underlying vector space of \(HP^n\subseteq Gr(2,2n+2)\). We have an inclusion \(k:HP^n\longrightarrow HP^{n+1}\) defined by
{\[k\left(\left(\begin{array}{c}x_1,\ldots,x_{2n+2}\\y_1,\ldots,y_{2n+2}\end{array}\right)\right)=\left(\begin{array}{c}x_1,\ldots,x_{n+1},0,x_{n+2},\ldots,x_{2n+2},0\\y_1,\ldots,y_{n+1},0,y_{n+2},\ldots,y_{2n+2},0\end{array}\right),\]} 
where \(\left(\begin{array}{c}v_1\\v_2\end{array}\right)\) means a two dimensional subspace written in its coordinates spanned by \(v_1, v_2\) in a \(k\)-vector space.

The following proposition follows by applying \cite[Theorem 8.1]{PW} and Proposition \ref{base} on the cohomology theory defined in Definition \ref{CPW}. {It was also proved in \cite[Corollary 1.2]{HW}.}
\begin{proposition}\label{Chow-Witt}
Let \(w:HP^n\to \mathrm{Spec}(k)\) be the structure map. Then the map
\[\begin{array}{ccc}A^{0,0}(pt)&\longrightarrow&A^{4i,2i}(HP^n)\\x&\longmapsto&w^*(x)\cdot p_1(\mathscr{U}^{\vee})^{i}\end{array}\]
is an isomorphism between abelian groups, where \(i=0,\ldots,n\).
\end{proposition}
\begin{theorem}\label{splitting}
For any $n\geq 0$, we have
\[\widetilde{\mathbb{Z}}_{pt}(HP^n)\cong\oplus_{i=0}^{n}\widetilde{\mathbb{Z}}_{pt}(2i)[4i]\]
in \(\widetilde{DM}^{eff,-}(pt)\).
\end{theorem}
\begin{proof}
We prove by induction. The statement is clearly true for \(n=0\). We thus suppose it's true for some \(n\geq 0\) and prove the result for $n+1$. Let then \[\theta:\widetilde{\mathbb{Z}}_{pt}(HP^n)\longrightarrow\oplus_{i=0}^n\widetilde{\mathbb{Z}}_{pt}(2i)[4i]\] be such an isomorphism.

We claim that the inclusion \(k:\widetilde{\mathbb{Z}}_{pt}(HP^n)\longrightarrow\widetilde{\mathbb{Z}}_{pt}(HP^{n+1})\) splits in \(\widetilde{DM}^{eff,-}(pt)\). Indeed, Proposition \ref{base} yield a commutative diagram in which the vertical homomorphisms are isomorphisms
{\[
\xymatrix@C=1.2em{
\mathrm{Hom}_{\widetilde{DM}^{eff,-}(pt)}(\widetilde{\mathbb{Z}}_{pt}(HP^{n+1}),\widetilde{\mathbb{Z}}_{pt}(HP^n))\ar[r]^-{k^*}\ar[d]_{\theta_*} & \mathrm{Hom}_{\widetilde{DM}^{eff,-}(pt)}(\widetilde{\mathbb{Z}}_{pt}(HP^n),\widetilde{\mathbb{Z}}_{pt}(HP^n))\ar[d]_{\theta_*}\\
\mathrm{Hom}_{\widetilde{DM}^{eff,-}(pt)}(\widetilde{\mathbb{Z}}_{pt}(HP^{n+1}),\oplus_{i=0}^n\widetilde{\mathbb{Z}}_{pt}(2i)[4i])\ar[r]^-{k^*} & \mathrm{Hom}_{\widetilde{DM}^{eff,-}(pt)}(\widetilde{\mathbb{Z}}_{pt}(HP^n),\oplus_{i=0}^n\widetilde{\mathbb{Z}}_{pt}(2i)[4i]).}
\]}
It suffices then to prove that for any \(i=0,2,\ldots,2n\), the pull-back 
\[
k^*:A^{2i,i}(HP^{n+1})\longrightarrow A^{2i,i}(HP^n)
\] 
is an isomorphism since the first horizontal arrow in the above diagram will be an isomorphism. This directly follows by Proposition \ref{Chow-Witt}.

Then by Proposition \ref{thomdefor}, Proposition \ref{deformation} and Proposition \ref{punctured affine space}, we have
{\[\widetilde{\mathbb{Z}}_{pt}(HP^{n+1})/\widetilde{\mathbb{Z}}_{pt}(HP^n)\cong\widetilde{M}_0(\mathbb{A}^{2n+2})\cong\widetilde{\mathbb{Z}}_{pt}(2n+2)[4n+4]\]}
in \(\widetilde{DM}^{eff,-}(pt)\), completing the induction process.

\end{proof}

Now we want to improve Theorem \ref{splitting} and find an explicit isomorphism using the first Pontryagin class of the dual tautological bundle on \(HP^n\).

\begin{lemma}\label{isomorphism}
Let \(\mathscr{C}\) be an additive category. Let \(M\), \(M_i\), \(i=1,\ldots,n\) be objects in \(\mathscr{C}\) such that \(Hom_{\mathscr{C}}(M_i, M_j)=0\) if \(i\neq j\). Suppose that there is an isomorphism \(\varphi:M\longrightarrow\oplus_i M_i\). Then for any morphism \(\varphi':M\longrightarrow\oplus_i M_i\), \(\varphi'\) is an isomorphism if and only if \(\varphi'_i\) is a free generator of \(Hom_{\mathscr{C}}(M, M_i)\) as left \(End_{\mathscr{C}}(M_i)\)-module for any \(i\), where \(\varphi'_i\) is composition of \(\varphi'\) and the \(i^{th}\) projection.
\end{lemma}
\begin{proof}
Suppose that \(\varphi'\) is an isomorphism.  We prove that \(\varphi'_i\) a free generator of \(Hom_{\mathscr{C}}(M, M_i)\) as a left \(End_{\mathscr{C}}(M_i)\)-module.

The action is free since \(\varphi'_i\) is surjective. Now suppose \(\psi\in Hom_{\mathscr{C}}(M, M_i)\). Since \(Hom_{\mathscr{C}}(M_i, M_j)=0\) if \(i\neq j\), we see that \(\psi=(\psi\circ\varphi'^{-1}\circ i_i)\circ(\varphi'_i)\) where \(i_i\) is the natural map as direct sum. Hence \(\psi\) can be generated by \(\varphi'_i\), so \(\varphi'_i\) is indeed a free generator.

Conversely, if we have a morphism \(\varphi':M\longrightarrow\oplus_i M_i\) such that \(\varphi'_i\) is a free generator of \(Hom_{\mathscr{C}}(M, M_i)\), then \(\varphi'_i=f_i\circ\varphi_i\) for some isomorphism \(f_i\). Hence \(\varphi'\) is also an isomorphism.
\end{proof}

For any \(S\in Sm/k\), we have a projection \(q:HP^n_S\longrightarrow HP^n\) and we set \(\mathscr{U}^{\vee}_S=q^*\mathscr{U}^{\vee}\).
\begin{theorem}\label{splitting1}
The map
\[\xymatrixcolsep{5pc}\xymatrix{\widetilde{\mathbb{Z}}_S(HP^n_S)\ar[r]^-{p_1(\mathscr{U}^{\vee}_S)^i}&\oplus_{i=0}^{n}\widetilde{\mathbb{Z}}_S(2i)[4i]}\]
is an isomorphism in \(\widetilde{DM}^{eff,-}(S)\).
\end{theorem}

\begin{proof}
Let's suppose at first \(S=pt\). By Theorem \ref{splitting}, Corollary \ref{cor:orthogonality} and Lemma \ref{isomorphism}, it remains to prove that \(p^i\) is a free generator of \(A^{4i,2i}(HP^n)\). 

{Denote the structure map \(HP^n\longrightarrow pt\) by \(w\), we see that the composite
\[\xymatrixcolsep{3pc}\xymatrix{A^{0,0}(pt)\ar[r]^-{w^*}&A^{0,0}(HP^n)\ar[r]^-{\cdot p_1(\mathscr{U}^{\vee})^i}&A^{4i,2i}(HP^n)}\]
is an isomorphism by Proposition \ref{Chow-Witt}  and it equals to the composite
\[\xymatrixcolsep{4pc}\xymatrix{A^{0,0}(pt)\ar[r]^-{\otimes\widetilde{\mathbb{Z}}_{pt}(2i)[4i]}&End(\widetilde{\mathbb{Z}}_{pt}(2i)[4i])\ar[r]^-{(p_1(\mathscr{U}^{\vee})^i)^*}&A^{4i,2i}(HP^n)}.\]
Hence the claim follows by Proposition \ref{cancellation}.}

In the case when \(S\) is general, we have a commutative diagram
\[
	\xymatrixcolsep{4pc}
	\xymatrix
	{
		q^*\widetilde{\mathbb{Z}}_{pt}(HP^n)\ar[r]^-{q^*(p_1(\mathscr{U}^{\vee}))}		&q^*\widetilde{\mathbb{Z}}_{pt}(2)[4]\\
		\widetilde{\mathbb{Z}}_S(HP^n_S)\ar[u]^-{\cong}\ar[r]^-{p_1(\mathscr{U}^{\vee}_S)}	&\widetilde{\mathbb{Z}}_S(2)[4]\ar[u]^{\cong}.
	}
\]
Hence the result follows by the commutative diagram
\[
	\xymatrixcolsep{5pc}
	\xymatrix
	{
		q^*\widetilde{\mathbb{Z}}_{pt}(HP^n)\ar[r]^-{q^*(p_1(\mathscr{U}^{\vee}))^i}	&\oplus_{i=0}^{n}q^*\widetilde{\mathbb{Z}}_{pt}(2i)[4i]\\
		\widetilde{\mathbb{Z}}_S(HP^n_S)\ar[r]^-{p_1(\mathscr{U}^{\vee}_S)^i}\ar[u]^{\cong}	&\oplus_{i=0}^{n}\widetilde{\mathbb{Z}}_S(2i)[4i],\ar[u]^{\cong}
	}
\]
where the upper horizontal arrow is an isomorphism by the discussion above.
\end{proof}

\begin{theorem}\label{quaternionic projective bundle theorem}
Let \(X\in Sm/S\) and let \((\mathscr{E},m)\) be a symplectic vector bundle of rank \(2n+2\) on \(X\). Let $\pi:HGr_X(\mathscr{E})\to X$ be the projection. Then, the map
\[\xymatrixcolsep{6pc}\xymatrix{\widetilde{\mathbb{Z}}_{S}(HGr_X(\mathscr{E}))\ar[r]^-{\pi\boxtimes p_1(\mathscr{U}^{\vee})^i}&\oplus_{i=0}^{n}\widetilde{\mathbb{Z}}_{S}(X)(2i)[4i]}\]
is an isomorphism in \(\widetilde{DM}^{eff,-}(S)\), funtorial for \(X\) in {\(Sm/S\)}.
\end{theorem}

\begin{proof}
{Suppose at first \(S=X\) and denote the morphism in the statement by \(\varphi\). We pick a finite open covering \(\{U_{\alpha}\}\) of \(X\) such that 
\[
(\mathscr{E},m)|_{U_{\alpha}}\cong\left(O_{U_{\alpha}}^{\oplus 2n+2},\left(\begin{array}{cc}&I\\-I&\end{array}\right)\right)
\]
for every \(\alpha\) (see \cite[Lemma 2.7]{A2}). For every \(\alpha\),  \(\varphi|_{U_{\alpha}}\) is an isomorphism by Theorem \ref{splitting1} and naturality of the first Pontryagin class (\cite[Definition 3.3]{A2}). So \(\varphi\) is an isomorphism by Proposition \ref{zarsep}.}

{In the case \(S\) is general, we see that \(p_{\#}(\varphi)\) is just the morphism in the statement, where \(p:X\longrightarrow S\) is the structure map.}
\end{proof}
\begin{remark}\label{pw}
The projective bundle theorem (Proposition \ref{pbt}) is false in \(\widetilde{DM}^{eff,-}(pt)\). Indeed, suppose that we have an isomorphism
\[\widetilde{\mathbb{Z}}_{pt}(\mathbb{P}^2)\cong\oplus_{i=0}^2\widetilde{\mathbb{Z}}_{pt}(i)[2i]\]
Then, applying \(A^{4,2}(-)\) on both sides, we find 
\[\widetilde{CH}^2(\mathbb{P}^2)\cong K_0^{MW}(k),\]
by Proposition \ref{base} and Lemma \ref{orthogonality}, contradicting \cite[Corollary 11.8]{F3}.
\end{remark}
\section{Gysin triangle}
\begin{theorem}\label{Thom}
{Let \(X\in Sm/S\) and \(E\) is an \(SL\)-orientable vector bundle over \(X\) with rank \(n\). Then
\[Th_S(E)\cong\widetilde{\mathbb{Z}}_S(X)(n)[2n]\]
in \(\widetilde{DM}^{eff,-}(S)\).}
\end{theorem}
\begin{proof}
{Suppose at first \(X=S\). The element \(th(E)\in A_X^{2n,n}(E)\) induces an element
\[\varphi\in Hom_{\widetilde{DM}^{eff,-}(X)}(Th_X(E),\widetilde{\mathbb{Z}}_X(n)[2n])\]
by adjunction. Suppose \(\{U_i\}\) is an open covering of \(X\) such that \(E|_{U_i}\cong O_{U_i}^{\oplus n}\) as \(SL\)-orientable bundles (see \cite[Lemma 2.7]{A2}). Then for every \(i\), \(\varphi|_{U_i}\) is just the composite
\[\xymatrix{Th_{U_i}(E|_{U_i})\ar[r]^-{\cong}&Th_{U_i}(O_{U_i}^{\oplus n})\ar[r]^{\cong}_{t_n}&\widetilde{\mathbb{Z}}_{U_i}(n)[2n]}\]
by Proposition \ref{SL}, which is an isomorphism. So by Proposition \ref{zarsep}, \(\varphi\) is an isomorphism.}

{For the case when \(S\) is general, we see that \(p_{\#}(\varphi)\) is just what we want, where \(p:X\longrightarrow S\) is the structure map.}
\end{proof}
\begin{theorem}\label{Thom1}
Let \(X\in Sm/k\) and \(E\neq 0\) is a vector bundle over \(X\) with rank \(n\). Then
\[Th_{pt}(E)\cong Th_{pt}(det(E))(n-1)[2n-2]\]
in \(\widetilde{DM}^{eff,-}(pt)\), where the right hand side only depends (up to an isomorphism) {on \(n\)} and the class of \(det(E)\) in \(Pic(X)/2Pic(X)\).
\end{theorem}
\begin{proof}
The idea of the proof comes from \cite[\S 4]{A2}. Since \(E\oplus det(E)^{\vee}\) is {\(SL\)-orientable}, we have isomorphisms
\[Th_X(E)\otimes_XTh_X(det(E)^{\vee})\cong Th_X(E\oplus det(E)^{\vee})\cong\widetilde{\mathbb{Z}}_X(n+1)[2n+2]\]
and
\[Th_X(det(E))\otimes_XTh_X(det(E)^{\vee})\cong\widetilde{\mathbb{Z}}_X(2)[4]\]
by Corollary \ref{additivity}, (3) and Theorem \ref{Thom}. Hence
\[Th_X(E)(2)[4]\cong Th_X(det(E))(n+1)[2n+2].\]
 Applying \(p_{\#}\) for the structure map \(p:X\longrightarrow pt\), we get 
\[Th_{pt}(E)(2)[4]\cong Th_{pt}(det(E))(n+1)[2n+2].\]
So the first statement follows by {Proposition \ref{cancellation}}. The second statement follows by \cite[Lemma 4.1]{A2}.
\end{proof}
Thus we have the following theorem.
\begin{theorem}\label{gysin triangle}
Let \(X\in Sm/S\) and \(Y\in Sm/S\) be a closed subscheme of \(X\) with \(\textrm{codim}(Y)=n\). 
\begin{enumerate}
\item
{If \(detN_{Y/X}\cong O_Y\), we have a distinguished triangle
\[\widetilde{\mathbb{Z}}_S(X\setminus Y)\longrightarrow\widetilde{\mathbb{Z}}_S(X)\longrightarrow \widetilde{\mathbb{Z}}_S(Y)(n)[2n]\longrightarrow\widetilde{\mathbb{Z}}_S(X\setminus Y)[1]\]
in \(\widetilde{DM}^{eff,-}(S)\).}
\item
If \(S=pt\) and \(n>0\), we have a distinguished triangle
\[\widetilde{\mathbb{Z}}_S(X\setminus Y)\longrightarrow\widetilde{\mathbb{Z}}_S(X)\longrightarrow Th_S(det(N_{Y/X}))(n-1)[2n-2]\longrightarrow\widetilde{\mathbb{Z}}_S(X\setminus Y)[1]\]
in \(\widetilde{DM}^{eff,-}(S)\), where the third term only depends (up to an isomorphism) on \(n\) and the class of \(det(N_{Y/X})\) in \(Pic(Y)/2Pic(Y)\).
\end{enumerate} 
\end{theorem}

\begin{proof}
By Proposition \ref{deformation to the normal bundle}, \(\widetilde{M}_Y(X)\cong Th_S(N_{Y/X})\). Now use Theorem \ref{Thom} and {Theorem \ref{Thom1}}.
\end{proof}
\textbf{Acknowledgments. }The author would like to thank his PhD advisor J. Fasel for giving me the basic idea of this article and helping during the subsequent research, and F. D\'eglise for helpful discussions. The careful work of the referee is also greatly appreciated.\\

\end{document}